\documentclass[11pt]{article}

\usepackage[T1]{fontenc}
\usepackage{lmodern}
\usepackage{amsmath,amssymb,amsthm,mathrsfs}
\usepackage{microtype}
\usepackage{enumitem}

\providecommand{\texorpdfstring}[2]{#1}

\theoremstyle{plain}
\newtheorem{theorem}{Theorem}[section]
\newtheorem{proposition}[theorem]{Proposition}
\newtheorem{lemma}[theorem]{Lemma}
\newtheorem{corollary}[theorem]{Corollary}
\numberwithin{equation}{section}

\theoremstyle{definition}
\newtheorem{definition}[theorem]{Definition}

\newcommand{\N}{\mathbb{N}}
\newcommand{\Fin}{\mathrm{Fin}}
\newcommand{\I}{\mathcal{I}}
\newcommand{\Istar}{\mathcal{I}^{\star}}
\newcommand{\Iplus}{\mathcal{I}^{+}}
\newcommand{\cl}{\operatorname{cl}}
\newcommand{\almost}{\subseteq^{\ast}}
\newcommand{\pminus}{P^{-}}
\newcommand{\HCF}{\mathsf{HCF}}

\title{\texorpdfstring{$\sigma$}{sigma}-hereditarily closure-preserving \texorpdfstring{$\mathcal I$}{I}-$sn$-networks, \texorpdfstring{$\mathcal I$}{I}-$sn$-metrizability, and local HCP-finiteness}
\author{Xing-Yu Hu\\
School of Mathematics and Statistics,\\
Hanjiang Normal University, Shiyan 442000, China\\
\texttt{huxingyu@hjnu.edu.cn}}
\date{}

\begin{document}
\maketitle

\begin{abstract}
For every admissible ideal $\I$ on $\N$, we prove that every regular space with a $\sigma$-hereditarily closure-preserving (HCP) $\I$-$sn$-network is $\I$-$sn$-metrizable, answering Problem~5.2 of Zhou--Liu--Liu--Lin. The proof combines existing characterizations with a local countability bound. If $y_n\to_\I x$ in a space $X$ and $\{n:y_n\ne x\}\in\Iplus$, then every point-discrete family of $\I$-sequential neighborhoods of $x$ is countable. If no such sequence exists at $x$, then $\{x\}$ itself is an $\I$-sequential neighborhood. These alternatives imply that every space with a $\sigma$-point-discrete $\I$-$sn$-network is $\I$-$snf$-countable.

We say that $\I$ has local HCP-finiteness if, in every $T_1$ space, the convergence and support assumptions above imply that each HCP family of $\I$-sequential neighborhoods of $x$ is finite. A countable Rudin--Keisler antichain of free ultrafilters gives an admissible ideal without this property, so the countability bound cannot in general be strengthened to finiteness. Every $\pminus$ ideal has local HCP-finiteness. We give sufficient conditions for Fubini products, heterogeneous Fubini sums, and increasing unions to preserve this property. These results establish local HCP-finiteness for the finite Fubini powers of $\Fin$, Kat\v{e}tov's first limit-stage ideal $\Fin^\omega$ and its finite successors, the tree-derived rank-$\omega$ ideal $\mathcal H_{<\omega}$, and the inductive-limit ideal $\Fin_\omega$. All these examples except $\Fin$ fail $\pminus$.
\end{abstract}


\section{Introduction}

Throughout, spaces are $T_1$ unless explicitly stated otherwise, and $\N=\{0,1,2,\ldots\}$. On any countably infinite carrier, $\Fin$ denotes the ideal of finite subsets of that carrier. An \emph{admissible ideal} $\I$ on $\N$ satisfies
\[
\Fin\subseteq\I\subsetneq\mathcal P(\N),
\]
and is closed under subsets and finite unions. Put
\[
\Istar=\{A\subseteq\N:\N\setminus A\in\I\},
\qquad
\Iplus=\mathcal P(\N)\setminus\I.
\]
A sequence $(x_n)$ in a space $X$ is \emph{$\I$-convergent} to $x\in X$, written $x_n\to_\I x$, if
\[
\{n\in\N:x_n\notin U\}\in\I
\]
for every neighborhood $U$ of $x$. Ideal convergence for sequences in metric spaces was introduced by Kostyrko, \v{S}al\'at and Wilczy\'nski \cite{KSW2001}. Lahiri and Das \cite{LahiriDas2005} extended it to topological spaces.

A set $P\subseteq X$ is an \emph{$\I$-sequential neighborhood} of $x$ if $x_n\to_\I x$ implies $\{n:x_n\notin P\}\in\I$ \cite[Definition~2.5]{ZhouLiu2020}. An \emph{$\I$-$sn$-network} is a family $\mathscr P=\bigcup_{x\in X}\mathscr P_x$ such that, for each $x\in X$, every member of $\mathscr P_x$ contains $x$, $\mathscr P_x$ is directed downward under finite intersections, $\mathscr P_x$ is a local network at $x$, and every member of $\mathscr P_x$ is an $\I$-sequential neighborhood of $x$ \cite[Definition~5.1]{ZLZ2023}.

Following Zhou--Lin--Zhang \cite[Definition~4.1]{ZLZ2023}, a space is \emph{$\I$-$snf$-countable} if each point has a countable local network consisting of $\I$-sequential neighborhoods. Zhou--Liu--Liu--Lin call a space \emph{$\I$-$sn$-metrizable} if it is regular and has a $\sigma$-locally finite $\I$-$sn$-network \cite[Definition~2.4]{ZLLL2024}.

A family $\mathscr A$ of subsets of $X$ is \emph{closure-preserving} if $\cl(\bigcup\mathscr B)=\bigcup_{B\in\mathscr B}\cl B$ for every subfamily $\mathscr B\subseteq\mathscr A$. It is \emph{hereditarily closure-preserving} (HCP) if, for every choice of subpieces $H(A)\subseteq A$ $(A\in\mathscr A)$, the family $\{H(A):A\in\mathscr A\}$ is closure-preserving \cite[Definition~2.1]{GeShenYing2007}. A family $\mathscr A$ is \emph{point-discrete} if $\{a_A:A\in\mathscr A\}$ is closed discrete for every choice $a_A\in A$ $(A\in\mathscr A)$ \cite[p.~199]{LiuLinLi2012}. A family is \emph{$\sigma$-point-discrete} if it is a countable union of point-discrete families. A family is $\sigma$-HCP if it is a countable union of HCP families. Every subfamily of an HCP family is HCP\@. In a $T_1$ space every HCP family is point-discrete. For every singleton selection, the selected singleton family is closure-preserving. Since singletons are closed in a $T_1$ space, every subset of the selected set is closed in $X$, so the selected set is closed discrete.

Zhou--Liu--Liu--Lin posed the following problem \cite[Problem~5.2]{ZLLL2024}.

\begin{quote}\small
\textbf{Problem 5.2.} Let $X$ be a regular space with a $\sigma$-hereditarily closure-preserving $\I$-$sn$-network. Is $X$ $\I$-$sn$-metrizable?
\end{quote}

Suppose a point $x$ admits an $\I$-convergent sequence with $\I$-positive support away from $x$. Lemma~\ref{lem:point-discrete-countability} shows that every point-discrete family of $\I$-sequential neighborhoods of $x$ is countable. Theorem~\ref{thm:problem52} uses this bound to answer Problem~5.2 affirmatively.

Zhou, Liu and Lin studied topological spaces determined by $\I$-convergence \cite{ZhouLiuLin2020}. Lin investigated $\I_{sn}$-open sets, $\I$-neighborhood spaces and $\I$-quotient spaces \cite{Lin2021}. Zhou--Lin related $\I$-$cs$-networks to images of metric spaces under $\I$-covering maps \cite{ZhouLin2022}. Liu, Lin and Zhou obtained metric-image characterizations for spaces defined by ideal convergence \cite{LiuLinZhou2024}.

Lin introduced ordinary $sn$-networks \cite{Lin1996}, and Lin--Yan used them in the study of sequence-covering maps of metric spaces \cite[Definition~4.2]{LinYan2001}. Ge studied $sn$-metrizability and countable $sn$-networks \cite{Ge2002,Ge2003,Ge2004}. Luo \cite{Luo2005} proved a mapping theorem for $sn$-metrizable spaces and compared several countability conditions on $sn$-networks. Ge--Lin characterized $g$-metrizable spaces by images of semi-metric spaces \cite{GeLin2007}, and Lin--Ge gave image characterizations of $sn$-metrizable spaces \cite{LinGe2019}. Lin--Yun \cite{LinYun2016} give general background on generalized metric spaces. Lin--Zhang \cite{LinZhang2024} survey point-countable covers, sequence-covering mappings and HCP families.

Classical work on HCP families in metrization and generalized metrization includes Burke--Engelking--Lutzer \cite{BEL1975} and Gruenhage \cite{Gruenhage1984}. For $\sigma$-HCP $k$-networks, Foged showed that a regular space is a closed image of a metric space exactly when it is Fr\'echet and has such a network \cite{Foged1985}. Tanaka showed that a regular space is $g$-metrizable exactly when it is weakly first countable and has such a network \cite{Tanaka1991}. Junnila--Yun gave a necessary and sufficient condition for a regular space with a $\sigma$-HCP $k$-network to be an $\aleph$-space \cite{JunnilaYun1992}. Liu \cite{Liu1993} surveys spaces with $\sigma$-HCP $k$-networks.

At the ordinary level, Ge--Shen--Ge \cite[Theorem~3.1]{GeShenYing2007} work throughout with regular $T_1$ spaces and prove $sn$-first countability for spaces with a $\sigma$-weakly HCP $sn$-network. The weak-HCP condition in their theorem is the point-discrete condition used here. Lin--Shen \cite{LinShen2010} later prove the stronger conclusion that every Hausdorff space with a $\sigma$-point-discrete $sn$-network has a $\sigma$-compact-finite $sn$-network. Since every $\I$-$sn$-network is an ordinary $sn$-network \cite[Lemma~5.2]{ZLZ2023}, the Lin--Shen result gives ordinary $sn$-first countability for Hausdorff spaces with a $\sigma$-point-discrete $\I$-$sn$-network. This does not by itself give $\I$-$snf$-countability, because the resulting ordinary $sn$-network is not asserted to consist of $\I$-sequential neighborhoods.

To prove the countability bound, let $y_n\to_\I x$ with $S=\{n:y_n\ne x\}\in\Iplus$. For a point-discrete family $\mathscr H$ of $\I$-sequential neighborhoods of $x$, consider the indices $n\in S$ for which $y_n$ belongs to infinitely many members of $\mathscr H$. Point-discreteness forces this exceptional index set to belong to $\I$. For $n\in S$ outside this set, $y_n$ belongs to only finitely many members of $\mathscr H$, while for every $P\in\mathscr H$ the set $\{n:y_n\in P,\ y_n\ne x\}$ is $\I$-positive. Hence $\mathscr H$ is countable. If no $\I$-convergent sequence at $x$ has $\I$-positive support away from $x$, then $\{x\}$ itself is an $\I$-sequential neighborhood. Applying these two alternatives at each point gives $\I$-$snf$-countability for every space with a $\sigma$-point-discrete $\I$-$sn$-network.

For a regular space with a $\sigma$-HCP $\I$-$sn$-network, existing characterizations complete the metrization argument once $\I$-$snf$-countability is established. Zhou--Lin--Zhang \cite[Lemma~5.2]{ZLZ2023} show that every $\I$-$sn$-network is an ordinary $sn$-network. Hence a $\sigma$-HCP $\I$-$sn$-network is a $\sigma$-HCP ordinary $sn$-network. Ge's characterization \cite[Lemma~2.2]{Ge2003} gives ordinary $sn$-metrizability. Zhou--Liu--Liu--Lin \cite[Theorem~3.6]{ZLLL2024} proved that $\I$-$sn$-metrizability is equivalent to the conjunction of $\I$-$snf$-countability and ordinary $sn$-metrizability. Theorem~\ref{thm:problem52}(2) therefore gives an affirmative answer to Problem~5.2.

In Hausdorff spaces, Liu--Lin--Li's Lemma~17(1) \cite{LiuLinLi2012} states that a point-discrete family is finite whenever a single nontrivial convergent sequence is eventually contained in every member. For an $\I$-convergent sequence, the exceptional index sets associated with $\I$-sequential neighborhoods need only belong to $\I$ and can be infinite. The local lemma below gives the corresponding countability bound under this weaker control. Ge's earlier HCP subsequence lemma \cite[Lemma~2.4]{Ge2003} is another point-selection precursor.

For HCP families of $\I$-sequential neighborhoods of a point $x$ in a $T_1$ space, we then ask which ideals force finiteness whenever $x$ admits an $\I$-convergent sequence with $\I$-positive support away from $x$. We denote local HCP-finiteness of $\I$ by $\HCF(\I)$, as in Definition~\ref{def:HCF}. Theorem~\ref{thm:HCF-failure} constructs an admissible ideal $\I$ for which this property fails. The witness consists of a zero-dimensional Hausdorff space $X$, a point $p\in X$, an $\I$-convergent sequence to $p$ with $\I$-positive support away from $p$, and a countably infinite HCP family of $\I$-sequential neighborhoods of $p$. The construction uses a one-point topology defined by an ideal and a countable Rudin--Keisler antichain of free ultrafilters \cite[pp.~199--200]{ShelahRudin1978}.

We next seek sufficient conditions for local HCP-finiteness and study which ideal constructions preserve it. Every $\pminus$ ideal has local HCP-finiteness by a finite diagonal selection argument. Local HCP-finiteness can nevertheless persist through constructions that need not preserve $\pminus$. A Fubini product $\mathcal K\otimes\mathcal J$ has local HCP-finiteness whenever $\mathcal K$ has $\pminus$ and $\mathcal J$ has local HCP-finiteness. A heterogeneous version allows the row ideals to vary under a cross-row Kat\v{e}tov absorption condition. These results give $\HCF(\Fin^\omega)$ at Kat\v{e}tov's first limit stage \cite[p.~240]{Katetov1972} and $\HCF(\Fin^{\omega+r})$ at its finite successors.

A second permanence route comes from increasing unions. If $\mathcal J_0\subseteq\mathcal J_1\subseteq\cdots$ are admissible ideals on one countable carrier and their union is a proper ideal, then local HCP-finiteness passes to the union when cofinally many stages are both $K$-uniform and locally HCP-finite. Applied to the tree-derived hierarchy of Pelayo G\'omez \cite[Theorems~3.4 and~3.5]{PelayoGomez2026}, this gives $\HCF(\mathcal H_{<\omega})$. It also gives $\HCF(\Fin_\omega)$ for the inductive limit in the Debs--Saint Raymond hierarchy \cite[Section~6.2]{DebsSaintRaymond2009}. The finite Fubini powers $\Fin^{\otimes r}$ for $r\ge2$, the ideals $\Fin^{\omega+r}$ for $r<\omega$, $\mathcal H_{<\omega}$, and $\Fin_\omega$ fail $\pminus$. Together, the Fubini and increasing-union routes establish stability of local HCP-finiteness beyond the class of $\pminus$ ideals.

Section~\ref{sec:prelim} collects the ideal-theoretic facts used throughout. Section~\ref{sec:local} isolates universal point-discrete countability, resolves Problem~5.2, and separates that countability phenomenon from the stronger local HCP-finiteness property. Section~\ref{sec:fubini} develops the Fubini permanence route, while Section~\ref{sec:unions} develops the increasing-union route and its two rank-$\omega$ applications.

\section{Preliminaries}\label{sec:prelim}

\subsection{Ideals, convergence, and separation rank}

The general-topology conventions follow Engelking \cite{Engelking1989}. For $A,B\subseteq\N$, write $A\almost B$ if $A\setminus B$ is finite. When an ideal is defined on another countable set such as $\N\times\N$, it is transported to $\N$ by a fixed bijection. The ideal properties considered here are invariant under such reindexing.

For an ideal $\mathcal A$ on a countably infinite carrier $D$, write $\operatorname{dom}(\mathcal A)=D$ and $\mathcal A^+=\mathcal P(D)\setminus\mathcal A$, and let $\mathcal A^\star=\{D\setminus A:A\in\mathcal A\}$ denote its dual filter. The ideal $\mathcal A$ is \emph{admissible} if it contains every finite subset of $D$ and is proper. It is \emph{tall} if every infinite subset of $D$ contains an infinite member of $\mathcal A$. When working on the carrier $D$ itself, a $D$-indexed family $(x_d)_{d\in D}$ in a space $X$ is $\mathcal A$-convergent to $x\in X$ if $\{d\in D:x_d\notin U\}\in\mathcal A$ for every neighborhood $U$ of $x$. A set $P\subseteq X$ is an $\mathcal A$-sequential neighborhood of $x$ if $\{d\in D:x_d\notin P\}\in\mathcal A$ for every such $\mathcal A$-convergent family. This agrees with the preceding $\N$-indexed definitions after reindexing. For such an indexed family, its \emph{support away from $x$} is $\{d\in D:x_d\ne x\}$. Regard $\mathcal P(D)$ as $2^D$ via characteristic functions and equip it with the product topology. For an analytic ideal $\mathcal A$, its \emph{Borel separation rank} $\operatorname{rk}(\mathcal A)$ is the least $\alpha<\omega_1$ for which there is $S\in\boldsymbol\Sigma^0_{1+\alpha}$ with $\mathcal A\subseteq S$ and $S\cap\mathcal A^\star=\varnothing$. This is the dual-ideal form of Debs--Saint Raymond's separation rank \cite[Definition~3.1]{DebsSaintRaymond2009}, under the convention recalled by Kwela \cite[p.~1]{Kwela2021}.

The metrization argument in Section~\ref{subsec:countability} uses Zhou--Lin--Zhang's implication from $\I$-$sn$-networks to ordinary $sn$-networks.

\begin{lemma}[{\cite[Lemma~5.2]{ZLZ2023}}]\label{lem:Isn-implies-sn}
For every admissible ideal $\I$, every $\I$-$sn$-network is an ordinary $sn$-network.
\end{lemma}

\subsection{The \texorpdfstring{$P^{-}$}{P-} property}

Hru\v{s}\'ak, Meza-Alc\'antara, Th\"ummel and Uzc\'ategui \cite[Theorem~3.8]{HrusakEtAl2017} give a Kat\v{e}tov-order characterization of the standard $\pminus$ property. The equivalent partition form in Definition~\ref{def:pminus} appears in Camargo--Uzc\'ategui \cite[Theorem~3.1(v)]{CamargoUzcategui2018} and Uzc\'ategui Aylwin \cite[Theorem~8.2(v)]{Uzcategui2019}.

\begin{definition}\label{def:pminus}
An admissible ideal $\I$ has $\pminus$ if, whenever $S\in\Iplus$ is partitioned as
\[
S=\bigsqcup_{n\in\N}E_n,
\qquad E_n\in\I,
\]
there is $A\in\Iplus$, $A\subseteq S$, such that $A\cap E_n$ is finite for every $n$.
\end{definition}

\begin{lemma}\label{lem:pminus-finite-dual}
For an admissible ideal $\I$, the following are equivalent.
\begin{enumerate}[label=\textup{(\roman*)}]
\item $\I$ has $\pminus$.
\item For every $S\in\Iplus$ and every sequence $(G_n)$ in $\Istar$, there are finite sets
\[
F_n\subseteq S\cap G_n\qquad(n\in\N)
\]
such that $\bigcup_nF_n\in\Iplus$.
\end{enumerate}
\end{lemma}

\begin{proof}
Assume first that $\I$ has $\pminus$. Given $S\in\Iplus$ and $G_n\in\Istar$, replace $G_n$ by $G_0\cap\cdots\cap G_n$ and put $T_n=S\cap G_n$, $T=\bigcap_nT_n$. If $T\in\Iplus$, partition $T$ into finite sets $F_n$. Then $F_n\subseteq S\cap G_n$ and $\bigcup_nF_n=T$.

Suppose $T\in\I$. The sets
\[
E_0=(S\setminus T_0)\cup T,
\qquad
E_{n+1}=T_n\setminus T_{n+1}
\]
form a partition of $S$ into members of $\I$. By $\pminus$, choose $A\in\Iplus$, $A\subseteq S$, with $A\cap E_n$ finite for every $n$. Set $F_n=A\cap E_{n+1}$. Then $F_n\subseteq T_n\subseteq S\cap G_n$ and
\[
\bigcup_nF_n=A\setminus(A\cap E_0)\in\Iplus.
\]

Conversely, suppose the finite-dual form holds and write $S=\bigsqcup_nE_n$ with $S\in\Iplus$ and $E_n\in\I$. Put
\[
G_n=(\operatorname{dom}(\I)\setminus S)\cup\bigcup_{i\ge n}E_i.
\]
Then $G_n\in\Istar$. Choose finite $F_n\subseteq S\cap G_n$ with $A=\bigcup_nF_n\in\Iplus$. For each $m$,
\[
A\cap E_m\subseteq\bigcup_{n\le m}F_n,
\]
so $A\cap E_m$ is finite. This proves the partition condition in Definition~\ref{def:pminus}.
\end{proof}

\begin{lemma}\label{lem:block-pminus-obstruction}
Let $\mathcal A$ be an admissible ideal on a countably infinite set $D$, and let
\[
D=\bigsqcup_{m\in\N}C_m.
\]
Suppose $C_m\in\mathcal A$ for every $m$, and every $B\subseteq D$ such that $B\cap C_m$ is finite for every $m$ belongs to $\mathcal A$. Then $\mathcal A$ does not have $\pminus$.
\end{lemma}

\begin{proof}
Put $G_n=\bigcup_{m\ge n}C_m$. The complement of $G_n$ is a finite union of members of $\mathcal A$, so $G_n\in\mathcal A^\star$. If finite $F_n\subseteq G_n$ are chosen, then for fixed $m$ only $F_0,\ldots,F_m$ can meet $C_m$. Hence $\bigcup_nF_n$ meets every $C_m$ in a finite set and therefore belongs to $\mathcal A$. Applying Lemma~\ref{lem:pminus-finite-dual} with the positive set $D$ shows that $\mathcal A$ does not have $\pminus$.
\end{proof}

In the dual-filter formulation, Debs--Saint Raymond \cite[Lemma~7.4]{DebsSaintRaymond2009} obtain a stronger embedding conclusion under a more general finite-intersection hypothesis. Only the finite-dual consequence in Lemma~\ref{lem:block-pminus-obstruction} is used in Sections~\ref{sec:fubini} and~\ref{sec:unions}.

Two standard conditions stronger than $\pminus$ are the additive property $(AP)$ and $P^+$. The property $(AP)$ means that for every sequence $(E_k)$ in $\I$ there is $E\in\I$ with $E_k\almost E$ for every $k$. Equivalently, $\mathrm{AP}(\I,\Fin)$ holds \cite[Lemma~3.9 and Definition~3.10]{MS2011}. The property $P^+$ means that every decreasing sequence $(S_n)$ in $\Iplus$ has $S\in\Iplus$ with $S\almost S_n$ for all $n$ \cite[p.~2023]{HrusakEtAl2017}.

Both conditions imply $\pminus$. Let $S=\bigsqcup_nE_n\in\Iplus$ with $E_n\in\I$. Under $(AP)$, choose $E\in\I$ with $E_k\setminus E$ finite for every $k$. Then $A=S\setminus E\in\Iplus$ and $A\cap E_k$ is finite. Under $P^+$, put $B_k=\bigcup_{n\ge k}E_n$. Each $B_k$ belongs to $\Iplus$, since $S\setminus B_k$ is a finite union of members of $\I$. Choose $B\in\Iplus$ with $B\almost B_k$ for every $k$ and set $A=B\cap S$. Since $B\almost B_0=S$, the set $B\setminus S$ is finite, and hence $A\in\Iplus$. Also,
\[
A\cap E_k\subseteq B\setminus B_{k+1},
\]
so $A\cap E_k$ is finite.

\subsection{Fubini products and \texorpdfstring{$K$}{K}-uniformity}

For ideals $\mathcal K,\mathcal J$ on countable sets $D$ and $E$, respectively, their \emph{Fubini product} is the ideal on $D\times E$ defined by
\[
A\in\mathcal K\otimes\mathcal J
\quad\Longleftrightarrow\quad
\{d\in D:A_d\notin\mathcal J\}\in\mathcal K,
\]
where $A_d=\{e\in E:(d,e)\in A\}$ \cite{KwelaTryba2017}. If $\mathcal K$ and $\mathcal J$ are admissible, then so is $\mathcal K\otimes\mathcal J$. For $r\ge1$ and ideals $\mathcal I_1,\ldots,\mathcal I_r$, the product is right-associated. For $r=1$ the iterate is $\mathcal I_1$. For $r\ge2$, it is defined by
\[
\mathcal I_1\otimes\cdots\otimes\mathcal I_r
=\mathcal I_1\otimes(\mathcal I_2\otimes\cdots\otimes\mathcal I_r).
\]
When all factors are $\Fin$, write $\Fin^{\otimes r}$ for the $r$-fold Fubini product.

For ideals $\mathcal A$ on $X$ and $\mathcal B$ on $Y$, write $\mathcal A\le_K\mathcal B$ if there is a map $f:Y\to X$ such that $f^{-1}[C]\in\mathcal B$ for every $C\in\mathcal A$. For $A\subseteq X$, let $\mathcal A\!\upharpoonright A=\{C\cap A:C\in\mathcal A\}$, viewed as an ideal on $A$. Following Hru\v{s}\'ak \cite[p.~37]{Hrusak2011}, an ideal $\mathcal A$ is \emph{$K$-uniform} if
\[
\mathcal A\!\upharpoonright A\le_K\mathcal A
\qquad\text{for every }A\in\mathcal A^+.
\]
Zhou--Lin--Zhang \cite[Definition~5.3]{ZLZ2023} use the same condition for ideal convergence. An ideal $\mathcal A$ is \emph{homogeneous} if $\mathcal A\!\upharpoonright A\cong\mathcal A$ for every $A\in\mathcal A^+$ \cite[Definition~1.3]{KwelaTryba2017}. Every homogeneous ideal is $K$-uniform.

\section{Local point-discrete countability and local HCP-finiteness}\label{sec:local}

\subsection{Local point-discrete countability and Problem~5.2}\label{subsec:countability}

The next lemma gives a local countability bound for point-discrete families of $\I$-sequential neighborhoods. We use this bound to prove $\I$-$snf$-countability and answer Problem~5.2.

\begin{lemma}\label{lem:point-discrete-countability}
Let $\I$ be an admissible ideal, let $X$ be a space, not necessarily $T_1$, and let $x\in X$. Suppose
\[
y_n\to_\I x,
\qquad
S=\{n:y_n\ne x\}\in\Iplus.
\]
If $\mathscr H$ is a point-discrete family of $\I$-sequential neighborhoods of $x$, then $\mathscr H$ is countable.
\end{lemma}

\begin{proof}
For $n\in S$, put
\[
\mathscr H(n)=\{P\in\mathscr H:y_n\in P\},
\]
and let
\[
D=\{n\in S:\mathscr H(n)\text{ is infinite}\}.
\]
First, $D\in\I$. Suppose otherwise that $D\in\Iplus$. Since $\I$ is admissible, $D$ is infinite. Enumerate it without repetition as $D=\{n_k:k\in\N\}$. Recursively choose pairwise distinct
\[
P_k\in\mathscr H(n_k).
\]
This is possible because each $\mathscr H(n_k)$ is infinite and only finitely many members have been excluded at stage $k$.

Every member of $\mathscr H$ contains $x$, since the constant sequence at $x$ is $\I$-convergent. Select $x$ from $P_0$, select $y_{n_k}$ from $P_k$ for $k\ge1$, and select $x$ from every remaining member of $\mathscr H$. Point-discreteness then makes
\[
A=\{x\}\cup\{y_{n_k}:k\ge1\}
\]
a closed discrete set. Since $D\subseteq S$, none of the points $y_{n_k}$ equals $x$. Moreover, $D\setminus\{n_0\}\in\Iplus$ because $\{n_0\}\in\I$. Hence
\[
x\in\cl\{y_{n_k}:k\ge1\}.
\]
A neighborhood $U$ of $x$ disjoint from $\{y_{n_k}:k\ge1\}$ would imply
\[
D\setminus\{n_0\}\subseteq\{n:y_n\notin U\}\in\I,
\]
contrary to $D\setminus\{n_0\}\in\Iplus$. This contradicts the discreteness of $A$ at $x$. Thus $D\in\I$.

Now fix $P\in\mathscr H$. Since $P$ is an $\I$-sequential neighborhood of $x$,
\[
E_P=\{n:y_n\notin P\}\in\I.
\]
Put
\[
G_P=\{n\in S:y_n\in P\}=S\setminus E_P.
\]
The set $G_P$ belongs to $\Iplus$. Otherwise
\[
S=(S\cap E_P)\cup G_P\in\I,
\]
a contradiction. Since $D\in\I$, also $G_P\setminus D\in\Iplus$, so there is some $n\in S\setminus D$ with $y_n\in P$. Therefore
\[
P\in\mathscr H(n)
\]
for some $n\in S\setminus D$, and hence
\[
\mathscr H\subseteq\bigcup_{n\in S\setminus D}\mathscr H(n).
\]
Each family $\mathscr H(n)$ on the right is finite by the definition of $D$, and $S\setminus D$ is countable. Thus $\mathscr H$ is countable.
\end{proof}

\begin{theorem}\label{thm:problem52}
Let $\I$ be an admissible ideal.
\begin{enumerate}[label=\textup{(\arabic*)}]
\item Every space with a $\sigma$-point-discrete $\I$-$sn$-network is $\I$-$snf$-countable.
\item Every regular space with a $\sigma$-hereditarily closure-preserving $\I$-$sn$-network is $\I$-$sn$-metrizable.
\end{enumerate}
\end{theorem}

\begin{proof}
For \textup{(1)}, let
\[
\mathscr P=\bigcup_{m\in\N}\mathscr P_m
\]
be a $\sigma$-point-discrete $\I$-$sn$-network on $X$, where each $\mathscr P_m$ is point-discrete. Fix $x\in X$.

Suppose first that there is an $\I$-convergent sequence $y_n\to_\I x$ such that
\[
\{n:y_n\ne x\}\in\Iplus.
\]
For each $m$, the trace
\[
\mathscr P_m\cap\mathscr P_x
\]
is a point-discrete family of $\I$-sequential neighborhoods of $x$. Lemma~\ref{lem:point-discrete-countability} shows that every such trace is countable. Hence
\[
\mathscr P_x
=
\bigcup_{m\in\N}(\mathscr P_m\cap\mathscr P_x)
\]
is countable. Since $\mathscr P_x$ is already a downward directed local network consisting of $\I$-sequential neighborhoods, it is a countable local $\I$-$sn$-network at $x$.

Suppose instead that every sequence $y_n\to_\I x$ satisfies
\[
\{n:y_n\ne x\}\in\I.
\]
Then $\{x\}$ is an $\I$-sequential neighborhood of $x$. The singleton family $\{\{x\}\}$ is downward directed and is a local network at $x$, because $\{x\}$ is contained in every neighborhood of $x$. Thus $x$ again has a countable local $\I$-$sn$-network. Therefore $X$ is $\I$-$snf$-countable.

For \textup{(2)}, let $X$ be regular and let $\mathscr P$ be a $\sigma$-HCP $\I$-$sn$-network. By Lemma~\ref{lem:Isn-implies-sn}, the same family is a $\sigma$-HCP ordinary $sn$-network. Ge's characterization \cite[Lemma~2.2]{Ge2003} gives ordinary $sn$-metrizability. Since $X$ is $T_1$ and every HCP family in a $T_1$ space is point-discrete, part~\textup{(1)} gives $\I$-$snf$-countability. Zhou--Liu--Liu--Lin's characterization \cite[Theorem~3.6]{ZLLL2024} therefore implies that $X$ is $\I$-$sn$-metrizable.
\end{proof}

Theorem~\ref{thm:problem52}(2) answers Problem~5.2 of Zhou--Liu--Liu--Lin \cite[Problem~5.2]{ZLLL2024}. Part~(1) gives $\I$-$snf$-countability from a $\sigma$-point-discrete $\I$-$sn$-network. For HCP families of $\I$-sequential neighborhoods of $x$ in $T_1$ spaces, we next ask when the convergence and positive-support assumptions of Lemma~\ref{lem:point-discrete-countability} imply finiteness.

\subsection{Local HCP-finiteness can fail}\label{subsec:HCF-failure}

\begin{definition}\label{def:HCF}
An admissible ideal $\I$ has the \emph{local HCP-finiteness property}, denoted $\HCF(\I)$, if the following holds. Whenever $X$ is a $T_1$ space, $x\in X$, $\mathscr H$ is an HCP family of $\I$-sequential neighborhoods of $x$, and $y_n\to_\I x$ with
\[
S=\{n:y_n\ne x\}\in\Iplus,
\]
then $\mathscr H$ is finite.
\end{definition}

Definition~\ref{def:HCF} requires finiteness for HCP families, whereas Lemma~\ref{lem:point-discrete-countability} gives universal countability already for the broader class of point-discrete families. Ideal isomorphisms preserve local HCP-finiteness by transporting indexed sequences along the carrier bijection, and preserve $K$-uniformity by conjugating Kat\v{e}tov maps with that bijection.

\begin{theorem}\label{thm:HCF-failure}
There is an admissible ideal $\I$ on $\N$ for which $\HCF(\I)$ fails. More precisely, the failure is witnessed by a zero-dimensional Hausdorff space $X$, a point $p\in X$, an $\I$-convergent sequence to $p$ with $\I$-positive support away from $p$, and a countably infinite HCP family of $\I$-sequential neighborhoods of $p$.
\end{theorem}

\begin{proof}
Use the convention that $\mathcal V\le_{\mathrm{RK}}\mathcal U$ when $\mathcal V$ is the image of $\mathcal U$ under a map. Shelah's theorem, presented by Shelah and Rudin \cite[pp.~199--200]{ShelahRudin1978}, gives $2^{2^{\aleph_0}}$ pairwise Rudin--Keisler incomparable ultrafilter types on $\N$. The principal type is Rudin--Keisler below every ultrafilter type, so none of these types is principal. Choose countably many representatives, all free, transport them to pairwise disjoint countably infinite sets $D_0,D_1,\ldots$, and denote the resulting ultrafilters by
\[
\mathcal U_0,\mathcal U_1,\ldots.
\]
Put
\[
\mathcal K_m=\mathcal P(D_m)\setminus\mathcal U_m
\]
for $m\in\N$, let $D=\bigsqcup_{m\in\N}D_m$, and define ideals on $D$ by
\[
\mathcal K=
\left\{A\subseteq D:A\cap D_m\in\mathcal K_m\text{ for every }m\in\N\right\}
\]
and
\[
\I=
\left\{A\subseteq D:
\{m:A\cap D_m\notin\mathcal K_m\}\in\Fin
\right\}.
\]
Both families are ideals, and every finite $A\subseteq D$ belongs to both. Since $D_m\notin\mathcal K_m$ for every $m$, one has $D\notin\mathcal K$, while $\{m:D_m\notin\mathcal K_m\}=\N\notin\Fin$, so $D\notin\I$. Hence both ideals are admissible, and $\mathcal K\subseteq\I$. The carrier $D$ may be transported to $\N$ by a bijection, as in the conventions of Section~\ref{sec:prelim}.

Let
\[
X=D\cup\{p\}.
\]
Declare every point of $D$ isolated, and take
\[
V_A=\{p\}\cup(D\setminus A),
\qquad A\in\mathcal K,
\]
as a neighborhood base at $p$. This is a neighborhood base because $V_A\cap V_B=V_{A\cup B}$ and $\mathcal K$ is an ideal. Since $X\setminus V_A=A\subseteq D$ is open, every $V_A$ is clopen. For $d\in D$, the singleton $\{d\}$ is open by isolation and closed because $X\setminus\{d\}=V_{\{d\}}$ is open. Hence $X$ has a clopen base and is zero-dimensional. The disjoint open neighborhoods $V_{\{d\}}$ and $\{d\}$ separate $p$ from $d$, while distinct points of $D$ are separated by singleton neighborhoods. Thus $X$ is Hausdorff. The $D$-indexed family $y_d=d$ is $\I$-convergent to $p$, because
\[
\{d:y_d\notin V_A\}=A\in\mathcal K\subseteq\I
\]
for every $A\in\mathcal K$. After reindexing $D$ by a bijection with $\N$, this is a sequence in the original definition of $\I$-convergence. The $D$-indexed notation is retained to make the row structure explicit. Its support away from $p$ is all of $D$, and
\[
D\notin\I
\]
because every row $D_m$ is $\mathcal K_m$-positive.

For $n\in\N$, put
\[
P_n=\{p\}\cup\bigcup_{m\ge n}D_m.
\]
Since $P_n\setminus P_{n+1}=D_n\ne\varnothing$ for every $n$, the family $\{P_n:n\in\N\}$ is countably infinite. This family is HCP\@. Fix an arbitrary subfamily. For each selected $P_n$, choose $H_n\subseteq P_n$, and put $H_n=\varnothing$ for the remaining indices. Since every point of $D$ is isolated, closure preservation can fail only at $p$. If $p\in\cl H_n$ for some $n$, then $p$ belongs to both $\cl(\bigcup_n H_n)$ and $\bigcup_n\cl H_n$. Otherwise $p\notin\cl H_n$ for every $n$. Then $H_n\subseteq D$. For each $n$, choose a basic neighborhood $V_{A_n}$ of $p$ disjoint from $H_n$. Then $H_n\subseteq A_n\in\mathcal K$, and hence $H_n\in\mathcal K$. For a fixed $m$, only the finitely many indices $n\le m$ can contribute points from $D_m$. Hence
\[
\left(\bigcup_n H_n\right)\cap D_m
\]
is a finite union of members of $\mathcal K_m$ and therefore belongs to $\mathcal K_m$. Thus $\bigcup_nH_n\in\mathcal K$, so $p\notin\cl\bigcup_nH_n$. Closure preservation at the isolated points of $D$ is automatic.

It remains to show that every $P_n$ is an $\I$-sequential neighborhood of $p$. Let $z_d\to_\I p$ be an arbitrary $D$-indexed family and write $f(d)=z_d$. Fix $r\in\N$ and put
\[
E=f^{-1}[D_r].
\]
The claim is that $E\in\I$. Suppose not. Then
\[
M=\{m:E\cap D_m\in\mathcal U_m\}
\]
is infinite. For $m\in M$, put $E_m=E\cap D_m$ and let
\[
f_m:E_m\longrightarrow D_r
\]
be the restriction of $f$. Since $E_m\in\mathcal U_m$, the trace
\[
\mathcal U_m\upharpoonright E_m
=
\{A\cap E_m:A\in\mathcal U_m\}
\]
is an ultrafilter on $E_m$. Let
\[
\mathcal V_m
=
\{B\subseteq D_r:f_m^{-1}[B]\in\mathcal U_m\upharpoonright E_m\}
\]
be its image under $f_m$.

Fix $B\in\mathcal K_r$, and let $A_B\subseteq D$ have $r$th section $B$ and all other sections empty. Then $A_B\in\mathcal K$. Since $X\setminus V_{A_B}=A_B$ and $z_d\to_\I p$,
\[
f^{-1}[A_B]\in\I.
\]
For $m\in M$,
\[
f_m^{-1}[B]=f^{-1}[A_B]\cap D_m.
\]
Consequently, for all but finitely many $m\in M$,
\[
f_m^{-1}[B]\in\mathcal K_m.
\]
For each such $m$, since $E_m\in\mathcal U_m$, membership in the trace $\mathcal U_m\upharpoonright E_m$ agrees with membership in $\mathcal U_m$ for subsets of $E_m$. Hence $B\notin\mathcal V_m$ for all but finitely many $m\in M$. If $C\in\mathcal U_r$, then $D_r\setminus C\in\mathcal K_r$. Applying the preceding conclusion to $B=D_r\setminus C$ and using that each $\mathcal V_m$ is an ultrafilter gives $C\in\mathcal V_m$ for all but finitely many $m\in M$. Enumerate $M=\{m_k:k\in\N\}$ without repetition and regard $\beta D_r$ as the space of ultrafilters on the discrete set $D_r$. For $C\subseteq D_r$, write $\widehat C=\{\mathcal W\in\beta D_r:C\in\mathcal W\}$. The clopen sets $\widehat C$ with $C\in\mathcal U_r$ form a neighborhood base at $\mathcal U_r$ in $\beta D_r$, so
\[
\mathcal V_{m_k}\longrightarrow\mathcal U_r.
\]
Every convergent sequence in $\beta D_r$ is eventually constant \cite[Corollary~3.6.15]{Engelking1989}. Hence $\mathcal V_{m_k}=\mathcal U_r$ for all sufficiently large $k$.

Choose a sufficiently large $k$ with $m_k\ne r$. By construction, $\mathcal V_{m_k}$ is a Rudin--Keisler image of $\mathcal U_{m_k}\upharpoonright E_{m_k}$. Since $E_{m_k}\in\mathcal U_{m_k}$, the trace $\mathcal U_{m_k}\upharpoonright E_{m_k}$ is Rudin--Keisler equivalent to $\mathcal U_{m_k}$. The inclusion $E_{m_k}\hookrightarrow D_{m_k}$ maps the trace ultrafilter to $\mathcal U_{m_k}$, while any retraction $D_{m_k}\to E_{m_k}$ that fixes $E_{m_k}$ maps $\mathcal U_{m_k}$ to the trace. It follows that
\[
\mathcal U_r\le_{\mathrm{RK}}\mathcal U_{m_k},
\]
contrary to the choice of the ultrafilters. Therefore $f^{-1}[D_r]\in\I$ for every $r\in\N$.

We have
\[
X\setminus P_n=\bigcup_{r<n}D_r.
\]
The inverse image of this finite union belongs to $\I$, so $P_n$ is an $\I$-sequential neighborhood of $p$. The family $\{P_n:n\in\N\}$, together with the $\I$-convergent family $y_d=d$ constructed above, witnesses the failure of $\HCF(\I)$.
\end{proof}

Because the witness family in Theorem~\ref{thm:HCF-failure} is HCP, hence point-discrete, that theorem shows that the countability conclusion of Lemma~\ref{lem:point-discrete-countability} cannot in general be strengthened to finiteness. Local HCP-finiteness requires additional structure on the ideal.

\subsection{\texorpdfstring{$P^{-}$}{P-} implies local HCP-finiteness}\label{subsec:pminus-HCF}

\begin{lemma}\label{lem:pminus-HCF}
If $\I$ has $\pminus$, then $\HCF(\I)$ holds.
\end{lemma}

\begin{proof}
Suppose $\mathscr H$ is infinite and choose pairwise distinct $P_0,P_1,\ldots\in\mathscr H$. Let $y_n\to_\I x$ and $S=\{n:y_n\ne x\}\in\Iplus$. Since each $P_j$ is an $\I$-sequential neighborhood of $x$,
\[
G_j=\{n:y_n\in P_j\}\in\Istar.
\]
By Lemma~\ref{lem:pminus-finite-dual}, choose finite
\[
F_j\subseteq S\cap G_j
\]
such that $F=\bigcup_jF_j\in\Iplus$. Put
\[
H(P_j)=\{y_n:n\in F_j\}\subseteq P_j
\]
and set $H(P)=\varnothing$ for $P\in\mathscr H\setminus\{P_j:j\in\N\}$. Each $H(P_j)$ is finite and contained in $X\setminus\{x\}$, so $x\notin\cl H(P_j)$ because $X$ is $T_1$.

On the other hand, $F\in\Iplus$ and $y_n\to_\I x$ imply
\[
x\in\cl\{y_n:n\in F\}
=\cl\Bigl(\bigcup_jH(P_j)\Bigr).
\]
Hereditary closure preservation gives
\[
\cl\Bigl(\bigcup_jH(P_j)\Bigr)
=\bigcup_j\cl H(P_j),
\]
a contradiction.
\end{proof}

Thus $\pminus$ gives a basic sufficient condition for local HCP-finiteness. Sections~\ref{sec:fubini} and~\ref{sec:unions} show how the property persists beyond $\pminus$ through Fubini constructions and increasing unions.

\section{Fubini constructions}\label{sec:fubini}

\subsection{Fubini products}

Local HCP-finiteness is preserved by Fubini products with a $\pminus$ outer factor, although $\pminus$ need not be preserved.

\begin{lemma}\label{lem:product-descent}
Let $\mathcal K$ and $\mathcal J$ be admissible ideals. Every $(\mathcal K\otimes\mathcal J)$-sequential neighborhood of $x$ is a $\mathcal J$-sequential neighborhood of $x$.
\end{lemma}

\begin{proof}
Let $P$ be a $(\mathcal K\otimes\mathcal J)$-sequential neighborhood of $x$, and let $z_j\to_{\mathcal J}x$. Define the array $w_{i,j}=z_j$. For each neighborhood $U$ of $x$ the set
\[
\{(i,j):w_{i,j}\notin U\}
=\operatorname{dom}(\mathcal K)\times\{j:z_j\notin U\}
\]
belongs to $\mathcal K\otimes\mathcal J$, so $w_{i,j}\to_{\mathcal K\otimes\mathcal J}x$. Hence
\[
\operatorname{dom}(\mathcal K)\times\{j:z_j\notin P\}
\in\mathcal K\otimes\mathcal J.
\]
Since $\mathcal K$ is proper, this is possible only if $\{j:z_j\notin P\}\in\mathcal J$.
\end{proof}

\begin{theorem}\label{thm:fubini-lifting}
Let $\mathcal K$ and $\mathcal J$ be admissible ideals. If $\mathcal K$ has $\pminus$ and $\HCF(\mathcal J)$ holds, then
\[
\HCF(\mathcal K\otimes\mathcal J)
\]
holds.
\end{theorem}

\begin{proof}
Let $X$ be $T_1$, let $x\in X$, let $\mathscr H$ be an HCP family of $(\mathcal K\otimes\mathcal J)$-sequential neighborhoods of $x$, and let
\[
y_{i,j}\to_{\mathcal K\otimes\mathcal J}x
\]
with
\[
S=\{(i,j):y_{i,j}\ne x\}\in(\mathcal K\otimes\mathcal J)^+.
\]
Write $S_i=\{j:y_{i,j}\ne x\}$ and
\[
C=\{i:S_i\in\mathcal J^+\}.
\]
Then $C\in\mathcal K^+$. Suppose that $\mathscr H$ is infinite and choose pairwise distinct $P_0,P_1,\ldots\in\mathscr H$.

By Lemma~\ref{lem:product-descent}, each $P_n$ is a $\mathcal J$-sequential neighborhood of $x$. If for some $i\in C$ the row $(y_{i,j})_j$ is $\mathcal J$-convergent to $x$, then $S_i\in\mathcal J^+$ and $\HCF(\mathcal J)$ forces the HCP family $\{P_n:n\in\N\}$ to be finite, a contradiction.

No row indexed by $C$ can therefore be $\mathcal J$-convergent to $x$. For each $i\in C$ choose a neighborhood $U_i$ of $x$ such that
\[
B_i=\{j:y_{i,j}\notin U_i\}\in\mathcal J^+.
\]
For each $n$, since $P_n$ is a $(\mathcal K\otimes\mathcal J)$-sequential neighborhood,
\[
D_n=\{(i,j):y_{i,j}\notin P_n\}\in\mathcal K\otimes\mathcal J.
\]
Thus
\[
E_n=\{i:(D_n)_i\notin\mathcal J\}\in\mathcal K,
\qquad
G_n=\operatorname{dom}(\mathcal K)\setminus E_n\in\mathcal K^\star.
\]
By Lemma~\ref{lem:pminus-finite-dual} applied to $C\in\mathcal K^+$ and $(G_n)$, there are finite
\[
R_n\subseteq C\cap G_n
\]
with $R=\bigcup_nR_n\in\mathcal K^+$. For $i\in R_n$ put
\[
A_{n,i}=B_i\setminus(D_n)_i.
\]
Since $(D_n)_i\in\mathcal J$ and $B_i\in\mathcal J^+$, $A_{n,i}\in\mathcal J^+$. Define
\[
H(P_n)=\{y_{i,j}:i\in R_n,\ j\in A_{n,i}\}\subseteq P_n
\]
and set $H(P)=\varnothing$ for the remaining members of $\mathscr H$. Because $R_n$ is finite,
\[
V_n=\bigcap_{i\in R_n}U_i
\]
is a neighborhood of $x$, with the empty intersection interpreted as $X$, and $H(P_n)\cap V_n=\varnothing$. Hence $x\notin\cl H(P_n)$ for every $n$.

Let
\[
F=\bigcup_n\bigcup_{i\in R_n}\bigl(\{i\}\times A_{n,i}\bigr).
\]
For every $i\in R$ the section $F_i$ contains some $\mathcal J$-positive $A_{n,i}$. Therefore
\[
\{i:F_i\notin\mathcal J\}\supseteq R\in\mathcal K^+,
\]
so $F\notin\mathcal K\otimes\mathcal J$. The convergence of $(y_{i,j})$ gives
\[
x\in\cl\{y_{i,j}:(i,j)\in F\}
=\cl\Bigl(\bigcup_nH(P_n)\Bigr),
\]
contradicting hereditary closure preservation.
\end{proof}

Every Fubini product of two admissible ideals fails $\pminus$. After reindexing the outer domain as $\N$, let
\[
C_i=\{i\}\times\operatorname{dom}(\mathcal J).
\]
Each $C_i$ belongs to $\mathcal K\otimes\mathcal J$, and every set meeting each $C_i$ finitely has all sections in $\mathcal J$. Lemma~\ref{lem:block-pminus-obstruction} applies. Equivalently, after reindexing the countable carriers, $\Fin\otimes\Fin\le_K\mathcal K\otimes\mathcal J$, and the Kat\v{e}tov boundary for $\pminus$ \cite[Theorem~3.8(4)]{HrusakEtAl2017} gives the same conclusion.

\begin{corollary}\label{cor:finite-products}
Let $r\ge1$ and let $\mathcal I_1,\ldots,\mathcal I_r$ be admissible ideals. If $\mathcal I_k$ has $\pminus$ for $1\le k<r$ and $\HCF(\mathcal I_r)$ holds, then
\[
\HCF(\mathcal I_1\otimes\cdots\otimes\mathcal I_r)
\]
holds. If $r\ge2$, the product ideal does not have $\pminus$.
\end{corollary}

\begin{proof}
For $r=1$, the local HCP-finiteness assertion is the hypothesis. For $r>1$, apply Theorem~\ref{thm:fubini-lifting} inductively from the innermost factor outward. The preceding Fubini-product observation proves the last assertion.
\end{proof}

\begin{corollary}\label{cor:Fin-powers}
For every $r\ge1$, $\HCF(\Fin^{\otimes r})$ holds. For every $r\ge2$, $\Fin^{\otimes r}$ does not have $\pminus$.
\end{corollary}

\begin{proof}
For every sequence $(E_k)$ of finite sets, $E=\varnothing$ satisfies $E_k\almost E$ for all $k$. Thus $\Fin$ has $(AP)$ and hence $\pminus$. Lemma~\ref{lem:pminus-HCF} gives $\HCF(\Fin)$, so Corollary~\ref{cor:finite-products} applies.
\end{proof}

Thus local HCP-finiteness is not confined to $\pminus$ ideals.

\subsection{Heterogeneous Fubini sums}\label{subsec:heterogeneous}

The product theorem treats a fixed inner ideal. At limit stages such as $\Fin^\omega$, the inner ideal varies from row to row. The Fubini-sum notation follows Kwela \cite[Subsection~2.2]{Kwela2021}. Let $\mathcal K$ be an admissible ideal on a countably infinite set $I$, and let $\mathcal J_i$ be an admissible ideal on a countably infinite set $D_i$ for each $i\in I$. Regard $\bigsqcup_{i\in I}D_i$ as the tagged union $\bigcup_{i\in I}(\{i\}\times D_i)$ and, for $A\subseteq\bigsqcup_{i\in I}D_i$, put $A_i=\{d\in D_i:(i,d)\in A\}$. Define
\[
\mathcal S=\mathcal K\text{-}\sum_{i\in I}\mathcal J_i
\]
by
\[
A\in\mathcal S
\quad\Longleftrightarrow\quad
\{i\in I:A_i\notin\mathcal J_i\}\in\mathcal K.
\]
Under these assumptions, $\mathcal S$ is admissible. Heredity and finite-union closure follow from the ideal axioms. Every finite set has only finite sections, so all of its sections lie in the corresponding row ideals. The whole carrier is excluded because its bad-row set is $I\notin\mathcal K$. For $i\in I$ and $E\in\mathcal J_i^+$, set
\[
R(i,E)=\{j\in I:\mathcal J_i\!\upharpoonright E\le_K\mathcal J_j\}.
\]
The global cross-row absorption condition is
\begin{equation}\label{eq:absorption}
R(i,E)\in\mathcal K^+
\qquad
(i\in I,\ E\in\mathcal J_i^+).
\end{equation}
For comparison with \eqref{eq:absorption}, fix ideals $\mathcal I$ and $\mathcal J$. Kwela--Lesner--Tryba \cite[Theorem~6.1]{KwelaLesnerTryba2026} prove that $\mathcal I\!\upharpoonright A\le_K\mathcal J$ for every $A\in\mathcal I^+$ if and only if every $\mathcal J$-continuous function preserves $\mathcal I$-convergence.

In the constant-row case $D_i=D$ and $\mathcal J_i=\mathcal J$ for all $i$, condition~\eqref{eq:absorption} reduces to $K$-uniformity of $\mathcal J$. Theorem~\ref{thm:fubini-lifting} does not require this additional uniformity assumption. Theorem~\ref{thm:heterogeneous-fubini} requires this absorption condition only for rows in some $W\in\mathcal K^\star$.

The next lemma provides the descent needed to adapt the finite-selection argument of Theorem~\ref{thm:fubini-lifting} to varying row ideals.

\begin{lemma}\label{lem:heterogeneous-descent}
Let $W\subseteq I$ and suppose $R(i,E)\in\mathcal K^+$ for every $i\in W$ and every $E\in\mathcal J_i^+$. Then every $\mathcal S$-sequential neighborhood of $x$ is a $\mathcal J_i$-sequential neighborhood of $x$ for each $i\in W$.
\end{lemma}

\begin{proof}
Fix $i\in W$, and suppose that an $\mathcal S$-sequential neighborhood $P$ is not a $\mathcal J_i$-sequential neighborhood of $x$. Choose $z_d\to_{\mathcal J_i}x$ such that
\[
E=\{d\in D_i:z_d\notin P\}\in\mathcal J_i^+.
\]
Then $R=R(i,E)\in\mathcal K^+$. For each $j\in R$, choose
\[
f_j:D_j\longrightarrow E
\]
witnessing $\mathcal J_i\!\upharpoonright E\le_K\mathcal J_j$, and define
\[
w_{j,e}=
\begin{cases}
z_{f_j(e)},&j\in R,\\
x,&j\notin R.
\end{cases}
\]
If $U$ is a neighborhood of $x$, then $B_U=\{d:z_d\notin U\}\in\mathcal J_i$, and for $j\in R$,
\[
\{e:w_{j,e}\notin U\}=f_j^{-1}[B_U\cap E]\in\mathcal J_j.
\]
For $j\notin R$ and $e\in D_j$, $w_{j,e}=x\in U$, so the corresponding bad section is empty. Hence $w\to_{\mathcal S}x$. On the other hand, for every $j\in R$ the whole $j$th row lies outside $P$. Thus the set of rows on which $w$ is outside $P$ is $\mathcal K$-positive, contradicting the $\mathcal S$-sequential-neighborhood property of $P$.
\end{proof}

\begin{theorem}\label{thm:heterogeneous-fubini}
Let
\[
\mathcal S=\mathcal K\text{-}\sum_{i\in I}\mathcal J_i.
\]
If $\mathcal K$ has $\pminus$ and, for some $W\in\mathcal K^\star$, $\HCF(\mathcal J_i)$ holds for every $i\in W$ and $R(i,E)\in\mathcal K^+$ for every $i\in W$ and every $E\in\mathcal J_i^+$, then $\HCF(\mathcal S)$ holds.
\end{theorem}

\begin{proof}
Let $X$ be $T_1$, let $x\in X$, let $\mathscr H$ be an HCP family of $\mathcal S$-sequential neighborhoods of $x$, and let
\[
y_{i,d}\to_{\mathcal S}x,
\qquad
S=\{(i,d):y_{i,d}\ne x\}\in\mathcal S^+.
\]
Put
\[
S_i=\{d:y_{i,d}\ne x\},
\qquad
C=\{i:S_i\in\mathcal J_i^+\}\in\mathcal K^+.
\]
Since $W\in\mathcal K^\star$, also $C_W=C\cap W\in\mathcal K^+$.
Suppose $\mathscr H$ is infinite and choose pairwise distinct $P_0,P_1,\ldots\in\mathscr H$. By Lemma~\ref{lem:heterogeneous-descent}, every $P_n$ is a $\mathcal J_i$-sequential neighborhood of $x$ for every $i\in W$.

If some row $(y_{i,d})_{d\in D_i}$ with $i\in C_W$ is $\mathcal J_i$-convergent to $x$, then $S_i\in\mathcal J_i^+$ and $\HCF(\mathcal J_i)$ forces $\mathscr H$ to be finite. Hence no row indexed by $C_W$ is $\mathcal J_i$-convergent. For each $i\in C_W$, choose a neighborhood $U_i$ of $x$ such that
\[
B_i=\{d:y_{i,d}\notin U_i\}\in\mathcal J_i^+.
\]
For each $n$, since $P_n$ is an $\mathcal S$-sequential neighborhood,
\[
\Delta_n=\{(i,d):y_{i,d}\notin P_n\}\in\mathcal S.
\]
Thus
\[
E_n=\{i:(\Delta_n)_i\notin\mathcal J_i\}\in\mathcal K,
\qquad
G_n=I\setminus E_n\in\mathcal K^\star.
\]
The finite-dual form of $\pminus$ gives finite
\[
R_n\subseteq C_W\cap G_n
\]
such that $R=\bigcup_nR_n\in\mathcal K^+$. For $i\in R_n$, put
\[
A_{n,i}=B_i\setminus(\Delta_n)_i\in\mathcal J_i^+
\]
and define
\[
H(P_n)=\{y_{i,d}:i\in R_n,\ d\in A_{n,i}\}\subseteq P_n,
\]
with $H(P)=\varnothing$ for the remaining members of $\mathscr H$. Since $R_n$ is finite,
\[
V_n=\bigcap_{i\in R_n}U_i
\]
is a neighborhood of $x$, with the empty intersection interpreted as $X$, and $H(P_n)\cap V_n=\varnothing$. Hence $x\notin\cl H(P_n)$.

Let
\[
F=\bigcup_n\bigcup_{i\in R_n}\bigl(\{i\}\times A_{n,i}\bigr).
\]
Every $i\in R$ has $F_i\notin\mathcal J_i$, so $F\notin\mathcal S$. The $\mathcal S$-convergence of $y$ gives
\[
x\in\cl\{y_{i,d}:(i,d)\in F\}
=\cl\Bigl(\bigcup_nH(P_n)\Bigr),
\]
contradicting hereditary closure preservation. Therefore $\mathscr H$ is finite.
\end{proof}

If, for some $W\in\mathcal K^\star$, every $\mathcal J_i$ with $i\in W$ is $K$-uniform and
\begin{equation}\label{eq:stage-absorption}
\{j\in I:\mathcal J_i\le_K\mathcal J_j\}\in\mathcal K^+
\qquad(i\in W),
\end{equation}
then the absorption hypothesis of Theorem~\ref{thm:heterogeneous-fubini} holds by transitivity of the Kat\v{e}tov order.

\subsection{The first limit stage and its finite successors}

At Kat\v{e}tov's first limit stage, the heterogeneous criterion applies to the row decomposition of $\Fin^\omega$. Kat\v{e}tov \cite[p.~240]{Katetov1972} gave the original filter recursion. Kwela \cite[Subsection~2.2]{Kwela2021} gives the corresponding ideal formulation and notation. The superscript notation distinguishes this hierarchy from the inductive-limit ideal $\Fin_\omega$ of Section~\ref{sec:unions}. The hierarchy satisfies
\[
\Fin^1=\Fin,
\qquad
\Fin^{\alpha+1}=\Fin\otimes\Fin^\alpha
\quad(0<\alpha<\omega_1),
\]
and at a limit ordinal $0<\lambda<\omega_1$,
\[
\Fin^\lambda
=\mathcal I_\lambda\text{-}\sum_{0<\alpha<\lambda}\Fin^\alpha,
\]
where $\mathcal I_\lambda$ is the ideal of bounded subsets of $\lambda\setminus\{0\}$. At the first limit stage,
\begin{equation}\label{eq:Fin-sup-omega}
\Fin^\omega
=\Fin\text{-}\sum_{1\le n<\omega}\Fin^n.
\end{equation}
For $1\le n<\omega$, $\Fin^n\cong\Fin^{\otimes n}$. Corollary~\ref{cor:Fin-powers} gives local HCP-finiteness. Kwela--Tryba \cite[Example~2.3 and Proposition~2.9]{KwelaTryba2017} proved that $\Fin$ is homogeneous and that Fubini products preserve homogeneity. Hence each finite power $\Fin^n$ is homogeneous and therefore $K$-uniform. If $1\le n\le m<\omega$, the projection onto the last $n$ coordinates witnesses $\Fin^n\le_K\Fin^m$ \cite[Subsection~2.4]{Kwela2021}. For each $1\le n<\omega$, the set $\{m:1\le m<\omega,\ \Fin^n\le_K\Fin^m\}$ contains the tail $\{m:n\le m<\omega\}$ and is $\Fin$-positive. Taking $W=\{n:1\le n<\omega\}\in\Fin^\star$, condition~\eqref{eq:stage-absorption} holds for the outer ideal $\Fin$. Since $\Fin$ has $\pminus$, Theorem~\ref{thm:heterogeneous-fubini} applies at the limit stage.

\begin{corollary}\label{cor:Fin-sup-omega-block}
For every $r<\omega$,
\[
\HCF(\Fin^{\omega+r})
\]
holds. Moreover, $\Fin^{\omega+r}$ does not have $\pminus$.
\end{corollary}

\begin{proof}
Equation~\eqref{eq:Fin-sup-omega} and Theorem~\ref{thm:heterogeneous-fubini} give $\HCF(\Fin^\omega)$. The heterogeneous sum in~\eqref{eq:Fin-sup-omega} is admissible, and Fubini products preserve admissibility. For $s<\omega$, the successor recursion is
\[
\Fin^{\omega+(s+1)}=\Fin\otimes\Fin^{\omega+s}.
\]
Thus Theorem~\ref{thm:fubini-lifting} applies at each finite successor.

For the failure of $\pminus$ at $r=0$, reindex the outer rows of \eqref{eq:Fin-sup-omega} by $\N$ and let $C_m$ denote the $m$th tagged row. Each $C_m$ belongs to $\Fin^\omega$, and any set meeting each $C_m$ finitely has a finite section in every row and hence belongs to $\Fin^\omega$. Lemma~\ref{lem:block-pminus-obstruction} applies. For each $s<\omega$, the Fubini-product observation above applies to the displayed product $\Fin^{\omega+(s+1)}$.
\end{proof}

\section{Increasing unions and rank-\texorpdfstring{$\omega$}{omega} limit ideals}\label{sec:unions}

\subsection{The increasing-union theorem}

The second permanence route replaces Fubini rows by an increasing chain. Its key step is a cross-ideal version of $K$-uniform reindexing. Zhang--Zhang \cite[Theorem~2.5]{ZhangZhang2021} used the same-ideal form, and Zhou--Lin--Zhang \cite[Definition~5.3 and Lemma~5.5]{ZLZ2023} later applied the reindexing to $\I$-$sn$ structures. In Lemma~\ref{lem:Kuniform-descent}, the convergence ideal and the $K$-uniform stage ideal need not be the same. This descent supplies the sequential-neighborhood hypothesis needed to apply local HCP-finiteness at a stage of the increasing chain.

\begin{lemma}\label{lem:Kuniform-descent}
Let $\mathcal J\subseteq\mathcal I$ be admissible ideals on the same countable set. If $\mathcal J$ is $K$-uniform, then every $\mathcal I$-sequential neighborhood of $x$ is a $\mathcal J$-sequential neighborhood of $x$.
\end{lemma}

\begin{proof}
Let $P$ be an $\mathcal I$-sequential neighborhood of $x$, and suppose $z_n\to_{\mathcal J}x$. If
\[
D=\{n:z_n\notin P\}\in\mathcal J^+,
\]
$K$-uniformity gives a map $f:\operatorname{dom}(\mathcal J)\to D$ such that
\[
B\in\mathcal J\!\upharpoonright D
\quad\Longrightarrow\quad
f^{-1}[B]\in\mathcal J.
\]
Define $w_n=z_{f(n)}$. For every neighborhood $U$ of $x$, the set
\[
\{k:z_k\notin U\}\cap D
\]
belongs to $\mathcal J\!\upharpoonright D$, and hence
\[
\{n:w_n\notin U\}
=f^{-1}[\{k:z_k\notin U\}\cap D]\in\mathcal J\subseteq\mathcal I.
\]
Thus $w_n\to_{\mathcal I}x$. But $w_n\notin P$ for every $n$, contradicting that $P$ is an $\mathcal I$-sequential neighborhood. Hence $D\in\mathcal J$.
\end{proof}

By Kwela--Tryba \cite[Example~2.3 and Proposition~2.9]{KwelaTryba2017}, every $\Fin^{\otimes r}$ is homogeneous and hence $K$-uniform for $r\ge1$. An increasing union of admissible ideals on the same carrier is admissible whenever the union is proper.

\begin{theorem}\label{thm:union-lifting}
Let $\mathcal J_0\subseteq\mathcal J_1\subseteq\cdots$ be admissible ideals on the same countable set, and suppose
\[
\mathcal I=\bigcup_{n\in\N}\mathcal J_n
\]
is proper. If for every $N$ there is $n\ge N$ such that $\HCF(\mathcal J_n)$ holds and $\mathcal J_n$ is $K$-uniform, then $\HCF(\mathcal I)$ holds.
\end{theorem}

\begin{proof}
Passing to a cofinal subsequence of stages with both properties and reindexing it by $\N$ does not change the union. Thus every $\mathcal J_n$ may be assumed to have local HCP-finiteness and be $K$-uniform.

Let $X$ be $T_1$, let $x\in X$, let $\mathscr H$ be an HCP family of $\mathcal I$-sequential neighborhoods of $x$, and let
\[
y_k\to_{\mathcal I}x,
\qquad
S=\{k:y_k\ne x\}\in\mathcal I^+.
\]
Suppose $\mathscr H$ is infinite and choose pairwise distinct $P_0,P_1,\ldots\in\mathscr H$. Put
\[
D_n=\{k:y_k\notin P_n\}\in\mathcal I
\]
and choose $d_n$ with $D_n\in\mathcal J_{d_n}$.

The sequence must converge with respect to $\mathcal J_N$ for some $N$. Otherwise, for every $r$ there is a neighborhood $U$ of $x$ with $\{k:y_k\notin U\}\notin\mathcal J_r$. For each $n$ put $r_n=\max\{n,d_n\}$ and choose a neighborhood $U_n$ such that
\[
B_n=\{k:y_k\notin U_n\}\notin\mathcal J_{r_n}.
\]
Set $A_n=B_n\setminus D_n$. Since $D_n\in\mathcal J_{r_n}$, $A_n\notin\mathcal J_{r_n}$. Also
\[
\{y_k:k\in A_n\}\subseteq P_n\setminus U_n,
\]
so $x\notin\cl\{y_k:k\in A_n\}$.

Let $A=\bigcup_nA_n$. If $A\in\mathcal I$, choose $M$ with $A\in\mathcal J_M$ and then choose $n\ge M$. Since $r_n\ge n\ge M$,
\[
A_n\subseteq A\in\mathcal J_M\subseteq\mathcal J_{r_n},
\]
a contradiction. Hence $A\in\mathcal I^+$. The $\mathcal I$-convergence of $(y_k)$ gives
\[
x\in\cl\{y_k:k\in A\}
=\cl\Bigl(\bigcup_n\{y_k:k\in A_n\}\Bigr),
\]
contradicting hereditary closure preservation for the subpieces $\{y_k:k\in A_n\}$ of $P_n$. Therefore there is $N$ such that
\[
y_k\to_{\mathcal J_N}x.
\]

By Lemma~\ref{lem:Kuniform-descent}, every member of $\mathscr H$ is a $\mathcal J_N$-sequential neighborhood of $x$. Because $S\notin\mathcal I$, also $S\notin\mathcal J_N$. Thus $\HCF(\mathcal J_N)$ implies that $\mathscr H$ is finite, a contradiction.
\end{proof}

\subsection{The tree-derived ideal \texorpdfstring{$\mathcal H_{<\omega}$}{H<omega}}

The first rank-$\omega$ application of the increasing-union theorem is the tree-derived ideal $\mathcal H_{<\omega}$ of Pelayo G\'omez. Fix a tree partition $\langle A_s:s\in\N^{<\omega}\rangle$ of $\N$ as in \cite[Definition~2.1]{PelayoGomez2026}. For $X\subseteq\N$, put
\[
T_X^\infty=\{s\in\N^{<\omega}:|X\cap A_s|=\infty\}.
\]
For $T\subseteq\N^{<\omega}$ define
\[
D_{\Fin}(T)=\{s\in T:\{n:s\mathbin{{}^{\frown}}n\in T\}\text{ is infinite}\},
\]
and iterate this operator by $D_{\Fin}^0(T)=T$ and $D_{\Fin}^{n+1}(T)=D_{\Fin}(D_{\Fin}^n(T))$. Following \cite[Definition~3.1]{PelayoGomez2026}, set
\[
\mathcal H_n=\{X\subseteq\N:\varnothing\notin D_{\Fin}^n(T_X^\infty)\}
\qquad(n<\omega).
\]
The resulting ideals are independent, up to ideal isomorphism, of the chosen tree partition \cite[Proposition~2.4]{PelayoGomez2026}. Pelayo G\'omez \cite[Theorems~3.4 and~3.5]{PelayoGomez2026} proved
\[
\mathcal H_0\subsetneq\mathcal H_1\subsetneq\cdots,
\qquad
\mathcal H_n\cong\Fin^{\otimes(n+1)}
\quad(n<\omega),
\]
and
\[
\mathcal H_{<\omega}:=\bigcup_{n<\omega}\mathcal H_n.
\]
Pelayo G\'omez \cite[Proposition~3.8]{PelayoGomez2026} proved that $\mathcal H_{<\omega}$ is a proper tall $\boldsymbol\Sigma^0_\omega$ ideal of separation rank exactly $\omega$.

\begin{theorem}\label{thm:Hlimit}
The ideal $\mathcal H_{<\omega}$ has local HCP-finiteness and does not have $\pminus$.
\end{theorem}

\begin{proof}
The isomorphisms $\mathcal H_n\cong\Fin^{\otimes(n+1)}$ and the admissibility of finite Fubini products show that each $\mathcal H_n$ is admissible. Corollary~\ref{cor:Fin-powers} and ideal-isomorphism invariance give local HCP-finiteness of each $\mathcal H_n$. By Kwela--Tryba \cite[Example~2.3 and Proposition~2.9]{KwelaTryba2017}, $\Fin^{\otimes(n+1)}$ is homogeneous and hence $K$-uniform. Since $K$-uniformity is invariant under ideal isomorphism, $\mathcal H_n$ is $K$-uniform as well. Theorem~\ref{thm:union-lifting} gives $\HCF(\mathcal H_{<\omega})$.

For the failure of $\pminus$, let $\varphi:\N\times\N\to\N$ be an isomorphism from $\Fin\otimes\Fin$ onto $\mathcal H_1$ and put
\[
C_i=\varphi[\{i\}\times\N].
\]
Each $C_i$ belongs to $\mathcal H_1$. If $B\subseteq\N$ meets each $C_i$ finitely, then $\varphi^{-1}[B]$ has finite vertical sections, so $B\in\mathcal H_1\subseteq\mathcal H_{<\omega}$. Lemma~\ref{lem:block-pminus-obstruction} shows that $\mathcal H_{<\omega}$ does not have $\pminus$.
\end{proof}

\subsection{Support lifts and the inductive-limit ideal \texorpdfstring{$\Fin_\omega$}{Fin omega}}

The second rank-$\omega$ application is $\Fin_\omega$. Its finite stages are isomorphic to support lifts of finite Fubini powers. For an admissible ideal $\mathcal I$ on a countably infinite set $D$, define its \emph{support lift} to be $\mathcal I\times 0$ in the standard Fubini-sum notation \cite[Section~2.1]{BarbarskiEtAl2013}. Here $0=\{\varnothing\}$, which is not admissible. On $D\times\N$, the support lift is
\[
\widehat{\mathcal I}
=
\{A\subseteq D\times\N:\operatorname{supp}(A)\in\mathcal I\},
\qquad
\operatorname{supp}(A)=\{d:A_d\ne\varnothing\}.
\]
\begin{proposition}\label{prop:support-lift}
Let $\mathcal I$ be an admissible ideal on a countably infinite set $D$. Then $\widehat{\mathcal I}$ is an admissible ideal. Moreover,
\begin{enumerate}[label=\textup{(\alph*)}]
\item If $\HCF(\mathcal I)$ holds, then $\HCF(\widehat{\mathcal I})$ holds.
\item If $\mathcal I$ is $K$-uniform, then $\widehat{\mathcal I}$ is $K$-uniform.
\end{enumerate}
\end{proposition}

\begin{proof}
The support map is monotone under inclusion and satisfies $\operatorname{supp}(A\cup B)=\operatorname{supp}(A)\cup\operatorname{supp}(B)$. Hence $\widehat{\mathcal I}$ is hereditary and closed under finite unions. Finite sets have finite support, while the whole carrier has support $D\notin\mathcal I$. Thus $\widehat{\mathcal I}$ is admissible.

For \textup{(a)}, let $X$ be a $T_1$ space, let $x\in X$, let $\mathscr H$ be an HCP family of $\widehat{\mathcal I}$-sequential neighborhoods of $x$, and let $y_{d,r}\to_{\widehat{\mathcal I}}x$ satisfy
\[
S=\{(d,r):y_{d,r}\ne x\}\in\widehat{\mathcal I}^{+}.
\]
Then $T=\operatorname{supp}(S)\in\mathcal I^+$. For each $d\in T$ choose $r_d$ with $y_{d,r_d}\ne x$ and put $z_d=y_{d,r_d}$, while $z_d=x$ for $d\notin T$. For every neighborhood $U$ of $x$,
\[
\{d:z_d\notin U\}
\subseteq
\operatorname{supp}\{(d,r):y_{d,r}\notin U\}\in\mathcal I,
\]
so $z_d\to_{\mathcal I}x$ and $\{d:z_d\ne x\}=T\in\mathcal I^+$.

If $P$ is a $\widehat{\mathcal I}$-sequential neighborhood of $x$ and $w_d\to_{\mathcal I}x$, the constant-row array $v_{d,r}=w_d$ is $\widehat{\mathcal I}$-convergent to $x$. Hence
\[
\{d:w_d\notin P\}
=
\operatorname{supp}\{(d,r):v_{d,r}\notin P\}
\in\mathcal I.
\]
Thus every member of $\mathscr H$ is an $\mathcal I$-sequential neighborhood. Applying $\HCF(\mathcal I)$ to $(z_d)$ shows that $\mathscr H$ is finite, proving \textup{(a)}.

For \textup{(b)}, let $A\in\widehat{\mathcal I}^{+}$ and put $T=\operatorname{supp}(A)\in\mathcal I^+$. Choose $g:\operatorname{dom}(\mathcal I)\to T$ witnessing $\mathcal I\!\upharpoonright T\le_K\mathcal I$, and, for each $d$, fix $r_d$ with $(g(d),r_d)\in A$. Define $F:D\times\N\to A$ by
\[
F(d,n)=(g(d),r_d).
\]
If $E\in\widehat{\mathcal I}\!\upharpoonright A$, then $\operatorname{supp}(E)\in\mathcal I$ and $\operatorname{supp}(E)\subseteq T$. Also,
\[
\operatorname{supp}(F^{-1}[E])
\subseteq
g^{-1}[\operatorname{supp}(E)]\in\mathcal I.
\]
Hence $F^{-1}[E]\in\widehat{\mathcal I}$, proving \textup{(b)}.
\end{proof}

For finite $n\ge1$, Kwela's subscript convention agrees with Kat\v{e}tov's superscript convention, $\Fin_n=\Fin^n$. At the limit stage $\omega$, $\Fin_\omega$ and $\Fin^\omega$ denote different constructions \cite[Subsection~2.4]{Kwela2021}. In Kwela's notation, put
\[
D_\omega=\bigsqcup_{m\ge1}\N^m,
\]
and, for $1\le k\le m$, let $\pi_{k,m}:\N^m\to\N^k$ be the projection onto the last $k$ coordinates. Using Kwela's full-domain convention \cite[Definition~2.1 and Subsection~2.4]{Kwela2021}, define
\[
\mathcal L_k=
\left\{
M\subseteq D_\omega:
\exists B\in\Fin^{\otimes k}\qquad
\forall m\ge k\quad
\pi_{k,m}[M\cap\N^m]\subseteq B
\right\}.
\]
Then $\Fin_\omega=\bigcup_{k\ge1}\mathcal L_k$. Kwela \cite[Subsection~2.1 and Corollary~3.5]{Kwela2021} showed that this limit is an admissible ideal.

\begin{lemma}\label{lem:Finomega-stages}
For every $k\ge1$,
\[
\mathcal L_k\subseteq\mathcal L_{k+1}
\qquad\text{and}\qquad
\mathcal L_k\cong\widehat{\Fin^{\otimes k}}.
\]
\end{lemma}

\begin{proof}
If $M\in\mathcal L_k$ is witnessed by $B\in\Fin^{\otimes k}$, then $\N\times B\in\Fin^{\otimes(k+1)}$ witnesses $M$ at stage $k+1$. Hence $\mathcal L_k\subseteq\mathcal L_{k+1}$.

For the isomorphism, for $s\in\N^k$ put
\[
F_s=\bigsqcup_{m\ge k}\{t\in\N^m:\pi_{k,m}(t)=s\}.
\]
The $F_s$ form a partition of $\bigsqcup_{m\ge k}\N^m$ into countably infinite sets, and
\[
M\in\mathcal L_k
\quad\Longleftrightarrow\quad
\{s:M\cap F_s\ne\varnothing\}\in\Fin^{\otimes k}.
\]
The possibly empty lower block $\bigsqcup_{1\le m<k}\N^m$ is unrestricted. Fix $s_0\in\N^k$, map each $F_s$ for $s\ne s_0$ bijectively onto $\{s\}\times\N$, and map $F_{s_0}$ together with the lower block bijectively onto $\{s_0\}\times\N$. Under the resulting bijection, the support of the image of $M$ differs from $\{s:M\cap F_s\ne\varnothing\}$ by at most the singleton $\{s_0\}$. Since $\Fin^{\otimes k}$ contains all finite sets, this is an ideal isomorphism from $\mathcal L_k$ onto $\widehat{\Fin^{\otimes k}}$.
\end{proof}

\begin{theorem}\label{thm:Finomega-limit}
The ideal $\Fin_\omega$ has local HCP-finiteness, does not have $\pminus$, and has separation rank $\omega$.
\end{theorem}

\begin{proof}
By Corollary~\ref{cor:Fin-powers}, every $\Fin^{\otimes k}$ has local HCP-finiteness. Kwela--Tryba \cite[Example~2.3 and Proposition~2.9]{KwelaTryba2017} show that every finite Fubini power of $\Fin$ is homogeneous and hence $K$-uniform. Proposition~\ref{prop:support-lift} therefore shows that $\widehat{\Fin^{\otimes k}}$ is admissible, has local HCP-finiteness, and is $K$-uniform. Lemma~\ref{lem:Finomega-stages} and ideal-isomorphism invariance transfer these properties to every $\mathcal L_k$. Theorem~\ref{thm:union-lifting} gives
\[
\HCF(\Fin_\omega).
\]

For the failure of $\pminus$, let $\psi:\N^2\times\N\to D_\omega$ witness $\widehat{\Fin\otimes\Fin}\cong\mathcal L_2$ and put
\[
C_i=\psi[\{((i,j),r):j,r\in\N\}].
\]
For each $i$, the support of $\psi^{-1}[C_i]$ is contained in the row with first coordinate $i$, so $C_i\in\mathcal L_2$. If $B\subseteq D_\omega$ meets every $C_i$ finitely, then each vertical section of $\operatorname{supp}(\psi^{-1}[B])$ is finite. Hence $\psi^{-1}[B]\in\widehat{\Fin\otimes\Fin}$ and $B\in\mathcal L_2\subseteq\Fin_\omega$. Lemma~\ref{lem:block-pminus-obstruction} shows that $\Fin_\omega$ does not have $\pminus$.

Debs--Saint Raymond give separation rank $\omega$ at their first limit stage \cite[Theorem~6.5]{DebsSaintRaymond2009}. Kwela \cite[Subsection~2.4]{Kwela2021} discusses a domain ambiguity in the original limit-stage presentation, so the full-domain convention used here is checked separately. For the upper bound, the support map $A\mapsto\operatorname{supp}(A)$ is Baire class one as the pointwise limit of its continuous second-coordinate truncations. Every finite Fubini power has finite Borel complexity \cite[Proposition~6.4]{DebsSaintRaymond2009}, as also noted by Kwela \cite[p.~1]{Kwela2021}. Since preimages of finite-level Borel sets under Baire class one maps are again of finite Borel class, each support lift $\widehat{\Fin^{\otimes k}}$ is of finite Borel class. The carrier bijection in Lemma~\ref{lem:Finomega-stages} induces a homeomorphism of the corresponding power-set spaces, so every $\mathcal L_k$ is of finite Borel class. Hence
\[
\Fin_\omega=\bigcup_{k\ge1}\mathcal L_k\in\boldsymbol\Sigma^0_\omega.
\]
Thus $\Fin_\omega$ is analytic. As a proper ideal, it is disjoint from its dual filter, so it is itself a $\boldsymbol\Sigma^0_\omega$ separator and $\operatorname{rk}(\Fin_\omega)\le\omega$.

For the lower bound, fix $k\ge1$ and define $\Phi_k:2^{\N^k}\to2^{D_\omega}$ by
\[
\Phi_k(A)=\bigsqcup_{m\ge k}\pi_{k,m}^{-1}[A].
\]
Each coordinate of $\Phi_k(A)$ depends on a single coordinate of $A$, so $\Phi_k$ is continuous. If $A\in\Fin^{\otimes k}$, then $\Phi_k(A)\in\mathcal L_k\subseteq\Fin_\omega$, with witness $A$. If $A\in(\Fin^{\otimes k})^\star$, then $\N^k\setminus A\in\Fin^{\otimes k}$ and
\[
D_\omega\setminus\Phi_k(A)
=\bigsqcup_{1\le m<k}\N^m\;\sqcup\;\Phi_k(\N^k\setminus A)
\in\mathcal L_k,
\]
with witness $\N^k\setminus A$. Hence $\Phi_k(A)\in\Fin_\omega^\star$. Therefore, if $\alpha<\omega_1$ and $S\in\boldsymbol\Sigma^0_{1+\alpha}$ separates $\Fin_\omega$ from $\Fin_\omega^\star$, then $\Phi_k^{-1}[S]\in\boldsymbol\Sigma^0_{1+\alpha}$ separates $\Fin^{\otimes k}$ from $(\Fin^{\otimes k})^\star$. Hence
\[
\operatorname{rk}(\Fin^{\otimes k})\le\operatorname{rk}(\Fin_\omega).
\]
Debs--Saint Raymond's finite-stage computation \cite[Theorem~6.5]{DebsSaintRaymond2009} gives $\operatorname{rk}(\Fin^{\otimes k})=k$. Hence $\operatorname{rk}(\Fin_\omega)\ge k$ for every $k\ge1$, so $\operatorname{rk}(\Fin_\omega)\ge\omega$, and equality follows.
\end{proof}

\section{Conclusion}\label{sec:final}

Lemma~\ref{lem:point-discrete-countability} shows that, for every admissible ideal $\I$, a point-discrete family of $\I$-sequential neighborhoods of $x$ is countable whenever $x$ admits an $\I$-convergent sequence with $\I$-positive support away from $x$. If no such sequence exists at $x$, the singleton $\{x\}$ is an $\I$-sequential neighborhood. These two alternatives give $\I$-$snf$-countability for every space carrying a $\sigma$-point-discrete $\I$-$sn$-network. In the regular $\sigma$-HCP case, Ge's characterization of ordinary $sn$-metrizability \cite[Lemma~2.2]{Ge2003} and Zhou--Liu--Liu--Lin's characterization of $\I$-$sn$-metrizability \cite[Theorem~3.6]{ZLLL2024} then settle Problem~5.2 \cite[Problem~5.2]{ZLLL2024} affirmatively for every admissible ideal.

The countably infinite HCP family in Theorem~\ref{thm:HCF-failure} shows that the same convergence and positive-support assumptions do not ensure finiteness for every admissible ideal. The Fubini theorems show that local HCP-finiteness passes from an inner ideal to a product when the outer ideal has $\pminus$. The heterogeneous theorem allows the row ideals to vary under cross-row Kat\v{e}tov absorption. The increasing-union theorem gives a second permanence principle through $K$-uniform stages. As a result, local HCP-finiteness holds for the finite Fubini powers of $\Fin$, for $\Fin^{\omega+r}$ with $r<\omega$, and for the rank-$\omega$ ideals $\mathcal H_{<\omega}$ and $\Fin_\omega$. Among these ideals, $\Fin^{\otimes r}$ for $r\ge2$, every $\Fin^{\omega+r}$ with $r<\omega$, $\mathcal H_{<\omega}$, and $\Fin_\omega$ fail $\pminus$.

The remaining structural problem is to characterize the admissible ideals satisfying local HCP-finiteness. Two related tasks remain. One is to determine whether local HCP-finiteness admits a direct description in terms of standard selection or Kat\v{e}tov-theoretic properties of ideals. The second asks which hypotheses in the permanence theorems can be weakened.


\begin{thebibliography}{99}
\sloppy
\hbadness=3000

\bibitem{BarbarskiEtAl2013}
P.~Barbarski, R.~Filip\'ow, N.~Mro\.zek, and P.~Szuca,
\newblock When does the Kat\v{e}tov order imply that one ideal extends the other?,
\newblock \emph{Colloq. Math.} \textbf{130} (2013), no.~1, 91--102.
\newblock doi:10.4064/cm130-1-9.

\bibitem{BEL1975}
D.~K. Burke, R.~Engelking, and D.~Lutzer,
\newblock Hereditarily closure-preserving collections and metrization,
\newblock \emph{Proc. Amer. Math. Soc.} \textbf{51} (1975), 483--488.
\newblock doi:10.1090/S0002-9939-1975-0370519-6.

\bibitem{CamargoUzcategui2018}
J.~Camargo and C.~Uzc\'ategui,
\newblock Selective separability on spaces with an analytic topology,
\newblock \emph{Topology Appl.} \textbf{248} (2018), 176--191.
\newblock doi:10.1016/j.topol.2018.09.002.

\bibitem{DebsSaintRaymond2009}
G.~Debs and J.~Saint Raymond,
\newblock Filter descriptive classes of Borel functions,
\newblock \emph{Fund. Math.} \textbf{204} (2009), no.~3, 189--213.
\newblock doi:10.4064/fm204-3-1.

\bibitem{Engelking1989}
R.~Engelking,
\newblock \emph{General Topology},
\newblock revised and completed edition, Sigma Series in Pure Mathematics, vol.~6,
Heldermann Verlag, Berlin, 1989.

\bibitem{Foged1985}
L.~Foged,
\newblock A characterization of closed images of metric spaces,
\newblock \emph{Proc. Amer. Math. Soc.} \textbf{95} (1985), no.~3, 487--490.
\newblock doi:10.1090/S0002-9939-1985-0806093-3.

\bibitem{Ge2002}
Y.~Ge,
\newblock On $sn$-metrizable spaces,
\newblock \emph{Acta Math. Sinica (Chin. Ser.)} \textbf{45} (2002), no.~2, 355--360 (in Chinese).
\newblock doi:10.12386/A2002sxxb0045.

\bibitem{Ge2003}
Y.~Ge,
\newblock Characterizations of $sn$-metrizable spaces,
\newblock \emph{Publ. Inst. Math. (Beograd) (N.S.)} \textbf{74(88)} (2003), 121--128.
\newblock doi:10.2298/PIM0374121G.

\bibitem{Ge2004}
Y.~Ge,
\newblock Spaces with countable $sn$-networks,
\newblock \emph{Comment. Math. Univ. Carolin.} \textbf{45} (2004), no.~1, 169--176.

\bibitem{GeLin2007}
Y.~Ge and S.~Lin,
\newblock $g$-metrizable spaces and the images of semi-metric spaces,
\newblock \emph{Czechoslovak Math. J.} \textbf{57(132)} (2007), no.~4, 1141--1149.
\newblock doi:10.1007/s10587-007-0117-x.

\bibitem{GeShenYing2007}
X.~Ge, J.~Shen, and Y.~Ge,
\newblock Spaces with $\sigma$-weakly hereditarily closure-preserving $sn$-networks,
\newblock \emph{Novi Sad J. Math.} \textbf{37} (2007), no.~1, 33--37.

\bibitem{Gruenhage1984}
G.~Gruenhage,
\newblock Generalized metric spaces,
\newblock in \emph{Handbook of Set-Theoretic Topology}, K.~Kunen and J.~E. Vaughan, eds.,
North-Holland, Amsterdam, 1984, pp.~423--501.
\newblock doi:10.1016/B978-0-444-86580-9.50013-6.

\bibitem{Hrusak2011}
M.~Hru\v{s}\'ak,
\newblock Combinatorics of filters and ideals,
\newblock in \emph{Set Theory and Its Applications}, Contemp. Math. \textbf{533},
Amer. Math. Soc., Providence, RI, 2011, pp.~29--69.
\newblock doi:10.1090/conm/533/10503.

\bibitem{HrusakEtAl2017}
M.~Hru\v{s}\'ak, D.~Meza-Alc\'antara, E.~Th\"ummel, and C.~Uzc\'ategui,
\newblock Ramsey type properties of ideals,
\newblock \emph{Ann. Pure Appl. Logic} \textbf{168} (2017), no.~11, 2022--2049.
\newblock doi:10.1016/j.apal.2017.06.001.

\bibitem{JunnilaYun1992}
H.~J.~K. Junnila and Z.~Yun,
\newblock $\aleph$-spaces and spaces with a $\sigma$-hereditarily closure-preserving $k$-network,
\newblock \emph{Topology Appl.} \textbf{44} (1992), 209--215.
\newblock doi:10.1016/0166-8641(92)90096-I.

\bibitem{Katetov1972}
M.~Kat\v{e}tov,
\newblock On descriptive classification of functions,
\newblock in \emph{General Topology and its Relations to Modern Analysis and Algebra, III},
Proceedings of the Third Prague Topological Symposium, 1971,
Academia, Prague, 1972, pp.~235--242.

\bibitem{KSW2001}
P.~Kostyrko, T.~\v{S}al\'at, and W.~Wilczy\'nski,
\newblock $\I$-convergence,
\newblock \emph{Real Anal. Exchange} \textbf{26} (2000/2001), no.~2, 669--686.
\newblock doi:10.2307/44154069.

\bibitem{Kwela2021}
A.~Kwela,
\newblock Inductive limits of ideals,
\newblock \emph{Topology Appl.} \textbf{300} (2021), Paper No.~107798, 13~pp.
\newblock doi:10.1016/j.topol.2021.107798.

\bibitem{KwelaLesnerTryba2026}
A.~Kwela, D.~Lesner, and J.~Tryba,
\newblock $\mathcal I$-closed sets and $\mathcal I$-continuous functions,
\newblock \emph{Topology Appl.} \textbf{390} (2026), Paper No.~109877.
\newblock doi:10.1016/j.topol.2026.109877.

\bibitem{KwelaTryba2017}
A.~Kwela and J.~Tryba,
\newblock Homogeneous ideals on countable sets,
\newblock \emph{Acta Math. Hungar.} \textbf{151} (2017), no.~1, 139--161.
\newblock doi:10.1007/s10474-016-0669-z.

\bibitem{LahiriDas2005}
B.~K. Lahiri and P.~Das,
\newblock $\I$- and $\I^{\ast}$-convergence in topological spaces,
\newblock \emph{Math. Bohem.} \textbf{130} (2005), no.~2, 153--160.
\newblock doi:10.21136/MB.2005.134133.

\bibitem{LinShen2010}
F.~C.~Lin and R.~X.~Shen,
\newblock Some notes on $\sigma$-point-discrete $sn$-networks,
\newblock \emph{Adv. Math. (China)} \textbf{39} (2010), no.~2, 212--216 (in Chinese).

\bibitem{Lin1996}
S.~Lin,
\newblock On sequence-covering $s$-mappings,
\newblock \emph{Adv. Math. (China)} \textbf{25} (1996), no.~6, 548--551 (in Chinese).

\bibitem{Lin2021}
S.~Lin,
\newblock On $\mathcal I$-neighborhood spaces and $\mathcal I$-quotient spaces,
\newblock \emph{Bull. Malays. Math. Sci. Soc.} \textbf{44} (2021), no.~4, 1979--2004.
\newblock doi:10.1007/s40840-020-01043-1.

\bibitem{LinGe2019}
S.~Lin and Y.~Ge,
\newblock Compact-covering and 1-sequence-covering images of metric spaces,
\newblock \emph{Houston J. Math.} \textbf{45} (2019), no.~1, 293--305.

\bibitem{LinYan2001}
S.~Lin and P.~Yan,
\newblock Sequence-covering maps of metric spaces,
\newblock \emph{Topology Appl.} \textbf{109} (2001), no.~3, 301--314.
\newblock doi:10.1016/S0166-8641(99)00163-7.

\bibitem{LinYun2016}
S.~Lin and Z.~Yun,
\newblock \emph{Generalized Metric Spaces and Mappings},
\newblock Atlantis Studies in Mathematics, vol.~6,
Atlantis Press, Paris, 2016.
\newblock doi:10.2991/978-94-6239-216-8.

\bibitem{LinZhang2024}
S.~Lin and J.~Zhang,
\newblock Recent progress on point-countable covers and sequence-covering mappings,
\newblock \emph{Axioms} \textbf{13} (2024), no.~10, Article~728.
\newblock doi:10.3390/axioms13100728.

\bibitem{Liu1993}
C.~Liu,
\newblock Spaces with a $\sigma$-hereditarily closure-preserving $k$-network,
\newblock \emph{Topology Proc.} \textbf{18} (1993), 179--188.

\bibitem{LiuLinLi2012}
C.~Liu, S.~Lin, and J.~Li,
\newblock Some properties on $\aleph_0$-weak bases,
\newblock \emph{Topology Proc.} \textbf{39} (2012), 195--208.

\bibitem{LiuLinZhou2024}
X.~Liu, S.~Lin, and X.~Zhou,
\newblock A study of spaces and mappings in the sense of ideal convergence,
\newblock \emph{Filomat} \textbf{38} (2024), no.~20, 7101--7110.
\newblock doi:10.2298/FIL2420101L.

\bibitem{Luo2005}
Z.~Luo,
\newblock $sn$-metrizable spaces and related matters,
\newblock \emph{Int. J. Math. Math. Sci.} \textbf{2005} (2005), no.~16, 2523--2531.
\newblock doi:10.1155/IJMMS.2005.2523.

\bibitem{MS2011}
M.~Ma\v{c}aj and M.~Sleziak,
\newblock $\I^{K}$-convergence,
\newblock \emph{Real Anal. Exchange} \textbf{36} (2010/2011), no.~1, 177--194.
\newblock doi:10.14321/realanalexch.36.1.0177.

\bibitem{PelayoGomez2026}
J.~de J. Pelayo G\'omez,
\newblock Tree-derived ideals: Fubini iterations, limit amalgamations, and Kat\v{e}tov obstructions,
\newblock arXiv:2607.16572v1 [math.LO], 2026.
\newblock doi:10.48550/arXiv.2607.16572.

\bibitem{ShelahRudin1978}
S.~Shelah and M.~E. Rudin,
\newblock Unordered types of ultrafilters,
\newblock \emph{Topology Proc.} \textbf{3} (1978), no.~1, 199--204.

\bibitem{Tanaka1991}
Y.~Tanaka,
\newblock $\sigma$-hereditarily closure-preserving $k$-networks and $g$-metrizability,
\newblock \emph{Proc. Amer. Math. Soc.} \textbf{112} (1991), no.~1, 283--290.
\newblock doi:10.1090/S0002-9939-1991-1049850-6.

\bibitem{Uzcategui2019}
C.~Uzc\'ategui Aylwin,
\newblock Ideals on countable sets: a survey with questions,
\newblock \emph{Rev. Integr. Temas Mat.} \textbf{37} (2019), no.~1, 167--198.
\newblock doi:10.18273/revint.v37n1-2019009.

\bibitem{ZhangZhang2021}
H.~Zhang and S.~Zhang,
\newblock Some applications of the theory of Kat\v{e}tov order to ideal convergence,
\newblock \emph{Topology Appl.} \textbf{301} (2021), Paper No.~107545.
\newblock doi:10.1016/j.topol.2020.107545.

\bibitem{ZhouLin2022}
X.~Zhou and S.~Lin,
\newblock On $\mathcal I$-covering images of metric spaces,
\newblock \emph{Filomat} \textbf{36} (2022), no.~19, 6621--6629.
\newblock doi:10.2298/FIL2219621Z.

\bibitem{ZLZ2023}
X.~Zhou, S.~Lin, and H.~Zhang,
\newblock $\I_{sn}$-sequential spaces and the images of metric spaces,
\newblock \emph{Topology Appl.} \textbf{327} (2023), Paper No.~108439.
\newblock doi:10.1016/j.topol.2023.108439.

\bibitem{ZLLL2024}
X.~Zhou, F.~Liu, L.~Liu, and S.~Lin,
\newblock $\I$-$sn$-metrizable spaces and the images of semi-metric spaces,
\newblock \emph{Open Mathematics} \textbf{22} (2024), no.~1, Article 20240053, 13~pp.
\newblock doi:10.1515/math-2024-0053.

\bibitem{ZhouLiu2020}
X.~Zhou and L.~Liu,
\newblock On $\I$-covering mappings and 1-$\I$-covering mappings,
\newblock \emph{J. Math. Res. Appl.} \textbf{40} (2020), no.~1, 47--56.
\newblock doi:10.3770/j.issn:2095-2651.2020.01.005.

\bibitem{ZhouLiuLin2020}
X.~Zhou, L.~Liu, and S.~Lin,
\newblock On topological spaces defined by $\mathcal I$-convergence,
\newblock \emph{Bull. Iran. Math. Soc.} \textbf{46} (2020), no.~3, 675--692.
\newblock doi:10.1007/s41980-019-00284-6.

\end{thebibliography}
\end{document}